\documentclass[12pt,a4paper]{article}
\usepackage[a4paper]{geometry}

\usepackage{datetime}

\usepackage[all]{xy}

\usepackage{amssymb, amsmath, amsthm}

\usepackage{eucal}

\usepackage{mathptmx}
\usepackage[scaled=.90]{helvet}
\usepackage{courier}

\usepackage[pdftex]{hyperref} % hyperindex=true, backref=page
\hypersetup{
  colorlinks = true,
  linkcolor = black,
  urlcolor = blue,
  citecolor = black,
}

\usepackage[
color=orange!80, % default color
bordercolor=black,
textwidth=.8in]
{todonotes}
\makeatletter \providecommand\@dotsep{5}
\makeatother
\newtheorem{theorem}{Theorem}
\newtheorem{proposition}[theorem]{Proposition}
\newtheorem{corollary}[theorem]{Corollary}

\theoremstyle{definition}
\newtheorem{definition}[theorem]{Definition}

\newtheorem{examples}[theorem]{Examples}
\newtheorem{remark}[theorem]{Remark}

\numberwithin{theorem}{section}
\numberwithin{equation}{section}

\renewcommand{\phi}{\varphi}

\newcommand{\Sbb}{\mathbb{S}}

\newcommand{\ZZ}{\mathbb{Z}}

\newcommand{\id}{\mathrm{id}}

\newcommand{\End}{\mathrm{End}}

\newcommand{\upB}{\mathrm{B}}
\newcommand{\upH}{\mathrm{H}}

\newcommand{\upS}{\mathrm{S}}

\newcommand{\gl}{\mathrm{gl}}
\newcommand{\GL}{\mathrm{GL}}

\newcommand{\calA}{\mathcal{A}}

\newcommand{\calC}{\mathcal{C}}
\newcommand{\calG}{\mathcal{G}}

\newcommand{\calM}{\mathcal{M}}

\newcommand{\calV}{\mathcal{V}}
\newcommand{\calW}{\mathcal{W}}

\newcommand{\bfk}{\mathbf{k}}
\newcommand{\bfK}{\mathbf{K}}

\newcommand{\ko}{\mathbf{ko}}

\newcommand{\Az}{\mathbf{Az}}

\newcommand{\Pic}{\mathrm{Pic}}
\newcommand{\Br}{\mathrm{Br}}
\newcommand{\Wh}{\mathrm{Wh}}
\newcommand{\Picspace}{\mathbf{Pic}}
\newcommand{\pic}{\mathbf{pic}}
\newcommand{\Brspace}{\mathbf{Br}}
\newcommand{\br}{\mathbf{br}}

\begin{document}

\title{Brauer spaces for commutative rings\\ and structured ring spectra}

\author{Markus Szymik}

\newdateformat{mydate}{\monthname\ \THEYEAR}
\mydate

\date{\today}

\maketitle

\renewcommand{\abstractname}{}

\begin{abstract}
\noindent
Using an analogy between the Brauer groups in algebra and the Whitehead groups in topology, we first use methods of algebraic K-theory to give a natural definition of Brauer spectra for commutative rings, such that their homotopy groups are given by the Brauer group, the Picard group and the group of units. Then, in the context of structured ring spectra, the same idea leads to two-fold non-connected deloopings of the spectra of units. 

%Natural maps relate these in the case of extensions and in the case of Eilenberg-Mac Lane spectra.

\vspace{\baselineskip}

\vspace{\baselineskip}
\end{abstract}

%%%%%%%%%%%%%%%%%%%%%%%%%%%%%%%%%%%%%%%%%%%%%%%%%%%%%%%%%%%%%%%%%%%%%

\section*{Introduction}

Let~$K$ be a field, and let us consider all finite-dimensional associative~$K$-algebras~$A$ up to isomorphism. For many purposes, a much coarser equivalence relation than isomorphism is appropriate: Morita equivalence. Recall that two such algebras~$A$ and~$B$ are called {\it Morita equi\-valent} if their categories of modules (or representations) are~$K$-equivalent. While this description makes it clear that we have defined an equivalence relation, Morita theory actually shows that there is a convenient explicit description of the relation that does not involve categories: two~$K$-algebras~$A$ and~$B$ are Morita equivalent if and only if there are non-trivial, finite-dimensional~$K$-vector spaces~$V$ and~$W$ such that~\hbox{$A\otimes_K\End_K(V)$} and~\hbox{$B\otimes_K\End_K(W)$} are isomorphic as~$K$-algebras. Therefore, we may say that Morita equivalence is generated by simple extensions: those from~$A$ to~$A\otimes_K\End_K(V)$. There is an abelian monoid structure on the set of Morita equivalence classes of algebras: The sum of the classes of two~$K$-algebras~$A$ and~$B$ is the class of the tensor product~$A\otimes_KB$, and the class of the ground field~$K$ is the neutral element. Not all elements of this abelian monoid have an inverse, but those algebras~$A$ for which the natural map~\hbox{$A\otimes_KA^\circ\to\End_K(A)$} from the tensor product with the opposite algebra~$A^\circ$ is an isomorphism clearly do. These are precisely the central simple~$K$-algebras. The abelian group of invertible classes of algebras is the Brauer group~$\Br(K)$ of~$K$. These notions have been generalized from the context of fields~$K$ to local rings by Azumaya~\cite{Azumaya}, and further to arbitrary commutative rings~$R$ by Auslander and Goldman~\cite{Auslander+Goldman}.

%%%%%%%%%%%%%%%%%%%%%%%%%%%%%%%%%%%%%%%%%%%%%%%%%%%%%%%%%%%%%%%%%%%%%

Our first aim is to use some of the classical methods of algebraic K-theory, recalled in the following Section~\ref{sec:background}, in order to define spaces~$\Brspace(R)$ such that their groups of components are naturally isomorphic to the Brauer groups~$\Br(R)$ of the commutative rings~$R$:
\[
	\pi_0\Brspace(R)\cong\Br(R).
\] 
Of course, this property does not characterize these spaces, so that we will have to provide motivation why the choice given here is appropriate. Therefore, in Section~\ref{sec:Whitehead}, we review Waldhausen's work~\cite{Waldhausen:Top1} and~\cite{Waldhausen:LNM} on the Whitehead spaces in geometric topology in sufficient detail so that it will become clear how this inspired our definition of the Brauer spaces~$\Brspace(R)$ to be given in Section~\ref{sec:Brauer_for_commutative_rings} in the case of commutative rings~$R$. Thereby we achieve the first aim. 

We can then relate the Brauer spaces to the classifying spaces of the Picard groupoids, and prove that we have produced a natural delooping of these, see Theorem~\ref{thm:Pic_identification_rings}. In particular, the higher homotopy groups are described by natural isomorphisms
\[
	\pi_1\Brspace(R)\cong\Pic(R)
\] 
and
\[
	\pi_2\Brspace(R)\cong\GL_1(R).
\] 

As will be discussed in Section~\ref{sec:Duskin}, we obtain an arguably more conceptual result than earlier efforts of Duskin~\cite{Duskin} and Street~\cite{Street}. This becomes particularly evident when we discuss comparison maps later. Another bonus of the present approach is the fact that it naturally produces infinite loop space structures on the Brauer spaces, so that we even have Brauer spectra~$\br(R)$ such that~\hbox{$\Omega^\infty\br(R)\simeq\Brspace(R)$} for all commutative rings~$R$.

%%%%%%%%%%%%%%%%%%%%%%%%%%%%%%%%%%%%%%%%%%%%%%%%%%%%%%%%%%%%%%%%%%%%%

The following Section~\ref{sec:Brauer_for_commutative_S-algebras} introduces Brauer spaces and spectra in the context of structured ring spectra. The approach presented here is based on and inspired by the same classical algebraic K-theory machinery that we will already have used in the case of commutative rings. These spaces and spectra refine the Brauer groups for commutative~$\Sbb$-algebras that have been defined in collaboration with Baker and Richter, see~\cite{Baker+Richter+Szymik}. Several groups of people are now working on Brauer groups of this type. For example, Gepner and Lawson are studying the situation of Galois extensions using methods from Lurie's higher topos theory~\cite{Lurie:HTT} and higher algebra. For connective rings, these methods are used in~\cite{Antieau+Gepner} for constructions and computations similar to~(but independent of) ours in this case. For example, the Brauer group~$\Br(\Sbb)$ of there sphere spectrum itself is known to be trivial. In contrast, this is not the case for the~(non-connective) chromatic localizations of the sphere spectrum by~\cite[Theorem~10.1]{Baker+Richter+Szymik}, and Angeltveit, Hopkins, and Lurie are making further progress towards the computation of these chromatic Brauer groups. 

%%%%%%%%%%%%%%%%%%%%%%%%%%%%%%%%%%%%%%%%%%%%%%%%%%%%%%%%%%%%%%%%%%%%%

In the final Section~\ref{sec:relative}, we indicate ways of how to apply the present theory. We first discuss the functoriality of our construction and define relative invariants which domi\-nate the relative invariants introduced in~\cite{Baker+Richter+Szymik}, see Proposition~\ref{prop:relative_relation}. Then we turn to the relation between the Brauer spectra of commutative rings and those of structured ring spectra. The Eilenberg-Mac Lane functor~$\upH$ produces structured ring spectra from ordinary rings, and it induces a homomorphism~\hbox{$\Br(R)\to\Br(\upH R)$} between the corresponding Brauer groups, see~\cite[Proposition 5.2]{Baker+Richter+Szymik}. We provide for a map of spectra which induces the displayed homomorphism after passing to homotopy groups, see~Proposition~\ref{prop:EMmap_spectra}.

Another potential application for Brauer spectra is already hinted at in~\cite{Clausen}: they are the appropriate target spectra for an elliptic~J-homomorphism. Recall that the usual~J-homomorphism can be described as a map $\ko\to\pic(\Sbb)$ from the real connective K-theory spectrum to the Picard spectrum of the sphere by means of the algebraic K-theory of symmetric monoidal categories. An elliptic~J-homomorphism should be a map from the connective spectrum of topological modular forms (or a discrete model thereof) to the Brauer spectrum of the sphere, where~$\br(\Sbb)\simeq\Sigma\pic(\Sbb)$ since~$\Br(\Sbb)=0$. It now seems highly plausible that such a map can be constructed from the algebraic K-theory of $\ko$, which has at least the correct chromatic complexity by results of Ausoni and Rognes~\cite{Ausoni+Rognes}. This will be pursued elsewhere.

\subsection*{Acknowledgment}

This work has been partially supported by the Deutsche Forschungsgemeinschaft~(DFG) through the Sonderforschungsbereich~(SFB)~701 ``Spektrale Strukturen und Topologische Methoden in der Mathe\-matik'' at Bielefeld University, and by the Danish National Research Foundation through the Centre for Symmetry and Deformation at the University of Copenhagen~(DNRF92). I would like to thank the referee of an earlier version for the detailed report.

%%%%%%%%%%%%%%%%%%%%%%%%%%%%%%%%%%%%%%%%%%%%%%%%%%%%%%%%%%%%%%%%%%%%%

\section{Background on algebraic K-theory}\label{sec:background}

After Quillen, algebraic K-theory is ultimately built on the passage from categorical input to topological output. Various different but equivalent methods to achieve this can be found in the literature. We will need to recall one such construction here, one that produces spectra from symmetric monoidal categories. This is originally due to Thomason~\cite{Thomason} and Shimada-Shimakawa~\cite{Shimada+Shimakawa}, building on earlier work of Segal~\cite{Segal} in the case where the monoidal structure is the categorical sum, and May~(\cite{May:perm1} and~\cite{May:perm2}) in the case when the category is permutative. See also~\cite{Thomason:Aarhus} and the appendix to~\cite{Thomason:CommAlg}. Here we will follow some of the more contemporary expositions such as the ones given in~\cite{Carlsson},~\cite{Elmendorf+Mandell},~\cite{Mandell} and~\cite{Dundas+Goodwillie+McCarthy}, for example. The reader familiar with Segal's~$\Gamma$-machine can safely skip this section and refer back to it for notation only.

%%%%%%%%%%%%%%%%%%%%%%%%%%%%%%%%%%%%%%%%%%%%%%%%%%%%%%%%%%%%%%%%%%%%

\subsection*{$\Gamma$-spaces}

One may say that the key idea behind Segal's machine is the insight that the combinatorics of abelian multiplications is governed by the category of finite pointed sets and pointed maps between them. A~{\it~$\Gamma$-space} is simply a pointed functor~$G$ from this category to the category of pointed spaces.

For set-theoretic purposes, it is preferable to work with a skeleton of the source category: For any integer~$n\geqslant0$, let~$n_+$ denote the pointed set~$(\{0,1,\dots,n\},0)$. We note that there is a canonical pointed bijection~\hbox{$1_+=\upS^0$} with the~$0$-sphere. The full subcategory category~$\Gamma^\circ$ has objects~$n_+$, for~$n\geqslant0$. These span a skeleton of the category of finite pointed sets that is isomorphic to the opposite of Segal's category~$\Gamma$. This explains the odd notation.

A~$\Gamma$-space~$G$ is called {\it special} if the Segal maps
\[
	G(N,n)\longrightarrow\prod_{N\setminus\{n\}} G(\upS^0)
\]
which are induced by the maps with singleton support, are weak equivalences. If this is the case, then there is an induced abelian monoid structure on the set~$\pi_0(G(\upS^0))$ of components of~$G(\upS^0)$, and the~$\Gamma$-space~$G$ is called {\it very special} if, in addition, this abelian monoid is a group.

For the purpose of exposition, let us note that basic examples of very special~$\Gamma$-spaces are given by abelian groups~$A$: We can associate to the finite pointed set~$n_+$ the space~$A^n$, and it should be thought of as the space of pointed maps~$f\colon n_+\to A$, where~$A$ is pointed by the zero element. For each pointed map~\hbox{$\phi\colon m_+\to n_+$} the transfer formula
\[
	(\phi_*f)(j)=\sum_{\phi(i)=j}f(i)
\]
induces a pointed map~\hbox{$\phi_*\colon A^m\to A^n$}.

%%%%%%%%%%%%%%%%%%%%%%%%%%%%%%%%%%%%%%%%%%%%%%%%%%%%%%%%%%%%%%%%%%%%

\subsection*{Spectra from~$\Gamma$-spaces}

Every~$\Gamma$-space~$G$ extends to a functor from spectra to spectra. In particular, every~$\Gamma$-space~$G$ has an associated spectrum~$G(\Sbb)$ by evaluation of this extension on the sphere spectrum~$\Sbb$, and hence it also has an associated (infinite loop) space~$\Omega^\infty G(\Sbb)$. For later reference, we need to recall some of the details from~\cite[Section~4]{Bousfield+Friedlander}. First of all, the functor~$G$ is extended to a functor from (all) pointed sets to pointed spaces by left Kan extension. Then it induces a functor from pointed simplicial sets to pointed simplicial spaces, and these can then be realized as spaces again. Finally, the~$n$-th space of~$G(\Sbb)$ is~$G(\upS^n)$, where~$\upS^n$ is the simplicial~$n$-sphere: the simplicial circle~$\upS^1=\Delta^1/\partial\Delta^1$ is a simplical set that has precisely~$n+1$ simplices of dimension~$n$, and~$\upS^n=\upS^1\wedge\dots\wedge\upS^1$ is the smash product of~$n$ simplicial circles.

We state the following fundamental theorem for later reference.

\begin{theorem}\label{thm:group_completion}
	{\bf\upshape(\cite[1.4]{Segal}, \cite[4.2,~4.4]{Bousfield+Friedlander})}
	If~$G$ is a special~$\Gamma$-space, then the adjoint structure maps
$G(\upS^n)\to\Omega G(\upS^{n+1})$ are equivalences for~$n\geqslant1$, and if~$G$ is very special, then this also holds for~$n=0$, so that under this extra hypothesis~$G(\Sbb)$ is an~$\Omega$-spectrum.
\end{theorem}

Thus, if~$G$ is a special~$\Gamma$-space, then the associated infinite loop space
is canonically identified with~$\Omega G(\upS^1)$, and if~$G$ is a very special~$\Gamma$-space, then the associated infinite loop space is canonically identified with~$G(\upS^0)$ as well.

%%%%%%%%%%%%%%%%%%%%%%%%%%%%%%%%%%%%%%%%%%%%%%%%%%%%%%%%%%%%%%%%%%%%

\subsection*{$\Gamma$-spaces from~$\Gamma$-categories}

In all the preceding definitions, spaces may be replaced by categories: A {\it~$\Gamma$-category} is a pointed functor~$\calG$ from the category of finite pointed sets to the category of pointed~(small) categories. A~$\Gamma$-category~$\calG$ is called {\it special} if the Segal functors are equivalences. If this is the case, then this induces an abelian monoid structure on the set of isomorphism classes of objects of~$\calG(\upS^0)$, and the~$\Gamma$-category~$\calG$ is called {\it very special} if, in addition, this abelian monoid is a group.

A~$\Gamma$-category~$\calG(-)$ defines a~$\Gamma$-space~$|\calG(-)|$ by composition with the geometric realization functor from~(small) categories to spaces.

%%%%%%%%%%%%%%%%%%%%%%%%%%%%%%%%%%%%%%%%%%%%%%%%%%%%%%%%%%%%%%%%%%%%

\subsection*{$\Gamma$-categories from symmetric monoidal categories}

Recall that the data required for a {\it monoidal category} are a category~$\calV$ together with a functor~\hbox{$\Box=\Box_\calV\colon\calV\times\calV\to\calV$}, the {\it product}, and an object~\hbox{$e=e_\calV$}, the {\it unit}, such that the product is associative and unital up to natural isomorphisms which are also part of the data. If, in addition, the product is commutative up to natural isomorphisms, then a choice of these turns~$\calV$ into a {\it symmetric monoidal category}.

Every symmetric monoidal category~$\calV$ gives rise to a~$\Gamma$-category~$\calV(-)$ as follows. Since the functor~$\calV(-)$ has to be pointed, we can assume that we have a pointed set~$(N,n)$ such that~$N\setminus\{n\}$ is not empty. Then, the objects of~$\calV(N,n)$ are the pairs~$(V,p)$, where~$V$ associates an object~$V(I)$ of~$\calV$ to every subset~$I$ of~$N\setminus\{n\}$, and~$p$ associates a map
\[
	p(I,J)\colon V(I\cup J)\longrightarrow V(I)\Box_\calV V(J)
\]
in~$\calV$ to every pair~$(I,J)$ of disjoint subsets~$I$ and~$J$ of~$N\setminus\{n\}$, in such way that four conditions are satisfied: the~$(V,p)$ have to be pointed, unital, associative, and symmetric. We refer to the cited literature for details.

\begin{examples}\label{ex:gamma012}
	If~$N=0_+$, then the category~$\calV(0_+)$ is necessarily trivial.
	If~$N=1_+$, then the category~$\calV(1_+)=\calV(\upS^0)$ is equivalent to original category~$\calV$ via the functor that evaluates at the non-base-point. If~$N=2_+$, then the category~$\calV(2_+)$ is equivalent to the category of triples $(V_{12},V_1,V_2)$ of objects together with morphisms $V_{12}\to V_1\Box_\calV V_2$.
\end{examples}

Let us now describe the functoriality of the categories~$\calV(N,n)$ in~$(N,n)$. If we are given a pointed map~$\alpha\colon(M,m)\to (N,n)$ of finite pointed sets, then a functor
\[
	\alpha_*=\calV(\alpha)\colon\calV(M,m)\longrightarrow\calV(N,n)
\]
is defined on the objects of~$\calV(M,m)$ by~$\alpha_*(V,p)=(\alpha_*V,\alpha_*p)$, where the components are defined by~\hbox{$(\alpha_*V)(I)=V(\alpha^{-1}I)$} and~\hbox{$(\alpha_*p)(I,J)=p(\alpha^{-1}I,\alpha^{-1}J)$}, and similarly on the morphisms of~$\calV(M,m)$. It is then readily checked that~$\alpha_*$ is a functor, and that the equations~\hbox{$\id_*=\id$} and~$(\alpha\beta)_*=\alpha_*\beta_*$ hold, so that we indeed have a~$\Gamma$-category~$\calV(-)\colon(N,n)\mapsto\calV(N,n)$.

%%%%%%%%%%%%%%%%%%%%%%%%%%%%%%%%%%%%%%%%%%%%%%%%%%%%%%%%%%%%%%%%%%%%

\subsection*{Algebraic K-theory}

If~$\calV$ is a symmetric monoidal category, then its {\it algebraic K-theory spectrum}~$\bfk(\calV)$ will be the spectrum associated with the~$\Gamma$-category that it determines. Its~$n$-th space is
\[
	\bfk(\calV)_n=|\calV(\upS^n)|.
\]
The {\it algebraic K-theory space}~$\bfK(\calV)=\Omega^\infty\bfk(\calV)$ of~$\calV$ is the underlying (infinite loop) space. By Theorem~\ref{thm:group_completion}, there is always a canonical equivalence
\[
	\bfK(\calV)\simeq\Omega|\calV(\upS^1)|.
\]
In addition, there is also a canonical equivalence
\[
	\bfK(\calV)\simeq|\calV(\upS^0)|\simeq|\calV|
\]
in the cases when the abelian monoid of isomorphism classes of~$\calV$ is a group under~$\Box_\calV$. This condition will be met in all the examples in this paper.

%%%%%%%%%%%%%%%%%%%%%%%%%%%%%%%%%%%%%%%%%%%%%%%%%%%%%%%%%%%%%%%%%%%%

\subsection*{Functoriality}

We finally need to comment on the functoriality of this algebraic K-theory construction. We would like maps~$\bfk(\calV)\to\bfk(\calW)$ of spectra to be induced by certain functors~\hbox{$\calV\to\calW$}.
It is clear that this works straightforwardly for strict functors between symmetric monoidal categories. But, it is useful to observe that it suffices, for example, to have an {\it op-lax symmetric monoidal functor} that is {\it strictly unital}: this is a functor~$F\colon\calV\to\calW$ such that~$F(e_\calV)=e_\calW$ holds, together with a natural transformation
\[
	\Phi\colon F(V\Box_\calV V')\longrightarrow F(V)\Box_\calW F(V')
\]
that commutes with the chosen associativity, unitality, and commutativity isomorphisms. Given such a strictly unital op-lax symmetric monoidal functor~$\calV\to\calW$, there is still an induced~$\Gamma$-functor~$\calV(-)\to\calW(-)$ between the associated~$\Gamma$-categories: It is defined on objects by the formula~$F_*(V,p)=(F(V),\Phi\circ F(p))$ that makes it clear how~$\Phi$ is used.

%%%%%%%%%%%%%%%%%%%%%%%%%%%%%%%%%%%%%%%%%%%%%%%%%%%%%%%%%%%%%%%%%%%%%

\section{Whitehead groups and Whithead spaces}\label{sec:Whitehead}

In this section, we will review just enough of Waldhausen's work on Whitehead spaces so that it will become clear how it inspired the definition of Brauer spaces to be given in the following Section~\ref{sec:Brauer_for_commutative_rings}.

A geometric definition of the Whitehead group of a space has been suggested by many people, see \cite{Stocker}, \cite{Eckmann+Maumary}, \cite{Siebenmann}, \cite{Farrell+Wagoner}, and \cite{Cohen}. We will review the basic ideas now. Ideally, the space~$X$ will be a nice topological space which has a universal covering, but it could also be a simplicial set if the reader prefers so. One considers finite cell extensions~(cofibrations)~$X\to Y$ up to homeomorphism under~$X$. An equivalence relation coarser than homeomorphism is generated by the so-called elementary extensions~\hbox{$Y\to Y'$}, or their inverses, the elementary collapses. By~\cite{Cohen:article}, this is the same as the equivalence relation generated by the simple maps. (Recall that a map of simplicial sets is {\it simple} if its geometric realization has contractible point inverses, see~\cite{Waldhausen+Jahren+Rognes}.) The sum of two extensions~$Y$ and~$Y'$ is obtained by glueing~$Y\cup_X Y'$ along~$X$, and~$X$ itself is the neutral element, up to homeomorphism. Not all elements have an inverse here, but those~$Y$ for which the structure map~$X\to Y$ is invertible (a homotopy equivalence) do. The abelian group of invertible extensions is called the {\it Whitehead group}~$\Wh(X)$ of~$X$. 

The preceding description of the Whitehead group, which exactly parallels the description of the Brauer group given in the introduction, makes it clear that these are very similar constructions. 

%%%%%%%%%%%%%%%%%%%%%%%%%%%%%%%%%%%%%%%%%%%%%%%%%%%%%%%%%%%%%%%%%%%%%

The Whitehead group~$\Wh(X)$ of a space~$X$, as described above, is actually a homotopy group of the Whitehead space. Let us recall from~\cite[Section 3.1]{Waldhausen:LNM} how this space can be constructed. We denote by~$\calC_X$ the category of the cofibrations under~$X$; the objects are the cofibrations~$X\to Y$ as above, and the morphisms from~$Y$ to~$Y'$ are the maps under~$X$. The superscript~$f$ will denote the subcategory of finite objects, where~$Y$ is generated by the image of~$X$ and finitely many cells. The superscript~$h$ will denote the subcategory of the invertible objects, where the structure map is an equivalence. The prefix~$s$ will denote the subcategory of simple maps. Then there is a natural bijection
\begin{equation}\label{eq:Wh_is_a_group}
	\Wh(X)\cong\pi_0|s\calC^{fh}_X|,
\end{equation}
see~\cite[3.2]{Waldhausen+Jahren+Rognes}. 

The bijection~\eqref{eq:Wh_is_a_group} is an isomorphism of groups if one takes into account the fact that the category~$\calC_X$ is symmetric monoidal: it has (finite) sums. This leads to a delooping of the space~$|s\calC^{fh}_X|$. Because the abelian monoid~$\pi_0|s\calC^{fh}_X|$ is already a group by~\eqref{eq:Wh_is_a_group}, Waldhausen deduces that there is a natural homotopy equivalence
 \begin{equation*}
	 |s\calC^{fh}_X|\simeq|s\calC^{fh}_X(\upS^0)|\simeq\Omega|s\calC^{fh}_X(\upS^1)|,
\end{equation*}
see~\cite[Proposition 3.1.1]{Waldhausen:LNM}, and he calls~$|s\calC^{fh}_X(\upS^1)|$ the {\it Whitehead space} of~$X$. Thus, the Whitehead space of~$X$ is a path connected space, whose fundamental group is isomorphic to the Whitehead group~$\Wh(X)$ of~$X$. 

\begin{remark}
One may find the convention of naming a space after its fundamental group rather odd, but the terminology is standard in geometric topology.
\end{remark}

Because the category~$s\calC^{fh}_X$ is symmetric monoidal, the algebraic K-theory machine that we have reviewed in Section~\ref{sec:background} can be used to produce a spectrum, the {\it Whitehead spectrum} of~$X$, such that the Whitehead space is its underlying infinite loop space.

Since it has proven to be very useful in geometric topology to have Whitehead spaces and spectra rather than just Whitehead groups, will we now use the analogy presented in this Section in order to define Brauer spaces and spectra as homotopy refinements of Brauer groups.

%%%%%%%%%%%%%%%%%%%%%%%%%%%%%%%%%%%%%%%%%%%%%%%%%%%%%%%%%%%%%%%%%%%%%

\section{Brauer spectra for commutative rings}\label{sec:Brauer_for_commutative_rings}

In this section, we will complete the analogy between Brauer groups and Whitehead groups by defining Brauer spaces and spectra in nearly the same way as we have described the Whitehead spaces and spectra in the previous section. Throughout this section, the letter~$R$ will denote an ordinary commutative ring, and~\cite[Chapter II]{Bass} will be our standard reference for the facts used from Morita theory. See also~\cite{Bass+Roy}.

%%%%%%%%%%%%%%%%%%%%%%%%%%%%%%%%%%%%%%%%%%%%%%%%%%%%%%%%%%%%%%%%%%%%%

\subsection*{The categories~$\calA_R$}

Let~$R$ be a commutative ring. Given such an~$R$, we will now define a category~$\calA_R$. The objects are the associative~$R$-algebras~$A$. It might be useful to think of the associative~$R$-algebra~$A$ as a placeholder for the~$R$-linear category~$\calM_A$ of right~$A$-modules. The morphisms~$A\to B$ in~$\calA_R$ will be the~$R$-linear functors~\hbox{$\calM_A\to\calM_B$}. Composition in~$\calA_R$ is composition of functors and identities are the identity functors, so that it is evident that~$\calA_R$ is a category. In fact, the category~$\calA_R$ is naturally enriched in spaces: The~$n$-simplices in the space of morphisms~\hbox{$A\to B$} are the functors~$\calM_A\times[n]\to\calM_B$. Here~$[n]$ denotes the usual poset category with object set~$\{0,\dots,n\}$ and standard order.

\begin{remark}\label{rem:higher_categories}
	It seems tempting to work in a setting where the morphism are given by bimodules instead of functors, as in~\cite[XII.7]{MacLane}. But, composition and identities are then given only up to choices, and this approach does not define a category in the usual sense. Compare Remark~\ref{rem:Street}.
\end{remark}

%%%%%%%%%%%%%%%%%%%%%%%%%%%%%%%%%%%%%%%%%%%%%%%%%%%%%%%%%%%%%%%%%%%%%

\subsection*{Decorated variants}

There are full subcategories~$\calA^f_R$ and~$\calA^h_R$ of~$\calA_R$, defined as follows. Recall the following characterization of faithful modules from~\cite[IX.4.6]{Bass}.

\begin{proposition}\label{prop:bass}
	For finitely generated projective~$R$-modules~$P$, the following are equivalent.\\
	{\upshape(1)} The~$R$-module~$P$ is a generator of the category of~$R$-modules.\\
	{\upshape(2)} The rank function~$\mathrm{Spec}(R)\to\ZZ$ of~$P$ is everywhere positive.\\
	{\upshape(3)} There is a finitely generated projective~$R$-module~$Q$ such that~\hbox{$P\otimes_RQ\cong R^{\oplus n}$} for some positive integer~$n$.
\end{proposition}

The full subcategory~$\calA^f_R$ consists of those~$R$-algebras~$A$ which, when considered as an~$R$-module, are finitely generated projective, and faithful in the sense of Proposition~\ref{prop:bass}. 

An~$R$-algebra~$A$ is in the full subcategory~$\calA^h_R$ if and only if the natural map
\begin{equation*}
	A\otimes_RA^\circ\longrightarrow\End_R(A)
\end{equation*}
is an isomorphism. We are mostly interested in the intersection
\begin{equation*}
	\calA^{fh}_R=\calA^f_R\cap\calA^h_R.
\end{equation*}

\begin{remark}\label{rem:Azumaya}
	An~$R$-algebra~$A$ lies in the full subcategory~$\calA^{fh}_R$ if and only if it is an Azumaya~$R$-algebra in the sense of~\cite{Auslander+Goldman}. 
\end{remark}

While~$f$ and~$h$ refer to restrictions on the objects, and therefore define full subcategories, the prefix~$s$ will indicate that we are considering less morphisms: Morphisms~\hbox{$A\to B$} in~$s\calA^{fh}_R$ are those functors~$\calM_A\to\calM_B$ which are~$R$-linear equivalences of categories: Morita equivalences.  For the higher simplices in~$s\calA^{fh}_R$, we require that they codify natural isomorphisms rather than all natural transformations. This implies that the mapping spaces are nerves of groupoids. In particular, they satisfy the Kan condition.

\begin{remark}
By Morita theory, the~$R$-linear equivalences~$\calM_A\to\calM_B$ of categories are, up to natural isomorphism, all of the form~\hbox{$X\longmapsto X\otimes_AM$} for some invertible~$R$-symmetric~$(A,B)$-bimodule~$M$. And conversely, all such bimodules define equivalences.
\end{remark}

%%%%%%%%%%%%%%%%%%%%%%%%%%%%%%%%%%%%%%%%%%%%%%%%%%%%%%%%%%%%%%%%%%%%%

\subsection*{A symmetric monoidal structure}

A symmetric monoidal structure on~$\calA_R$ and its decorated subcategories is induced by the tensor product
\begin{equation*}
	(A,B)\mapsto A\otimes_RB
\end{equation*}
of~$R$-algebras. The neutral object is~$R$. We note that the tensor product is not the categorical sum in~$\calA_R$, because the morphisms in that category are not just the algebra maps.

\begin{proposition}\label{prop:pi0_is_a_group_1}
	With the induced multiplication, the abelian monoid~$\pi_0|s\calA^{fh}_R|$ of isomorphism classes of objects is an abelian group.
\end{proposition}

\begin{proof}
	The elements of the monoid~$\pi_0|s\calA^{fh}_R|$ are represented by the objects of the category~$s\calA^{fh}_R$, and we have already noted that these are just the Azumaya algebras in the sense of~\cite{Auslander+Goldman}, see~Remark~\ref{rem:Azumaya}. Because each Azumaya algebra~$A$ satisfies~$A\otimes_RA^\circ\cong\End_R(A)$, we have~$[A]+[A^\circ]=[\End_R(A)]$ in~$\pi_0|s\calA^{fh}_R|$, so that~$[A^\circ]$ is an inverse to~$[A]$ in~$\pi_0|s\calA^{fh}_R|$ if there is a path from~$\End_R(A)$ to~$R$ in the category~$|s\calA^{fh}_R|$. But, by Proposition~\ref{prop:bass}, we know that~$A$ is a finitely generated projective generator in the category of~$R$-modules, so that the~$R$-algebras~$\End_R(A)$ and~$R$ are Morita equivalent. This means that there exists an~$R$-linear equivalence, and this gives rise to a~$1$-simplex which connects the two vertices. This shows that  the monoid~$\pi_0|s\calA^{fh}_R|$ is in fact a group.
\end{proof}

\begin{proposition}\label{prop:pi0_is_Brauer_1}
	The group~$\pi_0|s\calA^{fh}_R|$ is naturally isomorphic to the Brauer group~$\Br(R)$ of the commutative ring~$R$ in the sense of~{\upshape\cite{Auslander+Goldman}}.
\end{proposition}

\begin{proof}
	The elements in both groups have the same representatives, namely the Azumaya algebras, and the multiplications and units are also agree on those. Thus, it suffices to show that the equivalence relations agree for both of them. The equivalence relation in the Brauer group is generated by the simple extensions, and the equivalence relation in~$\pi_0|s\calA^{fh}_R|$ is generated by Morita equivalence. 
	
	We have already seen in the preceding proof that simple extensions are Morita equivalent. Conversely, if an algebra~$A$ is Morita equivalent to~$R$, then~$A$ is isomorphic to~$\End_R(P)$ for some finitely generated projective generator~$P$ of the category of~$R$-modules, so that~$A$ is a simple extension of~$R$, up to isomorphism.
\end{proof}

%%%%%%%%%%%%%%%%%%%%%%%%%%%%%%%%%%%%%%%%%%%%%%%%%%%%%%%%%%%%%%%%%%%%%

\subsection*{Brauer spaces and spectra}

The following definition is suggested by Proposition~\ref{prop:pi0_is_Brauer_1}.

\begin{definition}
Let~$R$ be a commutative ring. The space
\begin{equation*}
	\Brspace(R)=|s\calA^{fh}_R|
\end{equation*}
is called the {\it Brauer space} of~$R$. 
\end{definition}

By Proposition~\ref{prop:pi0_is_Brauer_1}, there is an isomorphism
\begin{equation}\label{eq:pi_0_Br}
	\pi_0\Brspace(R)\cong\Br(R)
\end{equation}
that is natural in~$R$.

As described in Section~\ref{sec:background}, the symmetric monoidal structure on~$s\calA^{fh}_R$ also gives rise to an algebraic K-theory spectrum.

\begin{definition}
	Let~$R$ be a commutative ring. The spectrum
	\begin{equation*}
		\br(R)=\bfk(s\calA^{fh}_R)
	\end{equation*}
	is called the {\it Brauer spectrum} of~$R$. 
\end{definition}

We will now spell out the relation between the Brauer space~$\Brspace(R)$ and the Brauer spectrum~$\br(R)$ in detail.

\begin{proposition}\label{prop:delooping_rings}
	 There are natural homotopy equivalences
	 \begin{equation*}
		\Omega^\infty\br(R)
		\simeq\Omega|s\calA^{fh}_R(\upS^1)|
		\simeq|s\calA^{fh}_R(\upS^0)|
		\simeq|s\calA^{fh}_R|=\Brspace(R).
	\end{equation*}
\end{proposition}

\begin{proof}
As explained in Section~\ref{sec:background}, the first and third of these generally hold for symmetric monoidal categories, and the second uses the additional information provided by Proposition~\ref{prop:pi0_is_a_group_1}, namely that the spaces on the right hand side are already group complete.
\end{proof}

In particular, we also have natural isomorphisms
\[
	\pi_0\br(R)\cong\Br(R),
\]
of abelian groups, so that we can recover the Brauer group as the~$0$-th homotopy group of a spectrum. We will now determine the higher homotopy groups thereof.

%%%%%%%%%%%%%%%%%%%%%%%%%%%%%%%%%%%%%%%%%%%%%%%%%%%%%%%%%%%%%%%%%%%%%

\subsection*{Higher homotopy groups}

Let us now turn our attention to the higher homotopy groups of the Brauer space (or spectrum) of a commutative ring~$R$. 

Recall that an~$R$-module~$M$ is called {\it invertible} if there is another~$R$-module~$L$ and an isomorphism~$L\otimes_RM\cong R$ of~$R$-modules. (We remark that, later on, the ring~$R$ might be graded, and then we will also have occasion to consider graded~$R$-modules, but we will also explicitly say so when this will be the case.) The {\it Picard groupoid} of a commutative ring~$R$ is the groupoid of invertible~$R$-modules and their isomorphisms. The realization~$\Picspace(R)$ of the Picard groupoid can have only two non-trivial homotopy groups:
the group of components
\begin{equation}\label{eq:pi_0_Pic}
	\pi_0\Picspace(R)\cong\Pic(R)
\end{equation}
is the Picard group~$\Pic(R)$ of~$R$. The fundamental groups of the Picard space~$\Picspace(R)$ are all isomorphic to the group of automorphisms of the~$R$-module~$R$, which is the group~$\GL_1(R)$ of units in~$R$.
\begin{equation}\label{eq:pi_1_Pic}
	\pi_1(\Picspace(R),R)\cong\GL_1(R)
\end{equation}
See~\cite[Cor on~p.~3]{Weibel:Azumaya}, for example.

The multiplication on the set of components comes from the fact that the Picard groupoid is symmetric monoidal with respect to the tensor product~$\otimes_R$. Since the isomorphism classes of objects form a group by~\ref{eq:pi_0_Pic}, this also implies that there is Picard spectrum~$\pic(R)$ such that~$\Picspace(R)\simeq\Omega^\infty\pic(R)$, and there is a connected delooping
\[
	\upB\Picspace(R)\simeq\Omega^\infty\Sigma\pic(R)
\]
such that its homotopy groups are isomorphic to those of~$\Picspace(R)$, but shifted up by one. 

\begin{theorem}\label{thm:Pic_identification_rings}
	The components of~$\Brspace(R)$ are all equivalent to~$\upB\Picspace(R)$.
\end{theorem}

\begin{proof}
All components of an infinite loop space such as~$\Brspace(R)\simeq\Omega^\infty\br(R)$ have the same homotopy type. 

Therefore, it suffices to deal with the component of the unit~$R$. But that component is the realization of the groupoid of~$R$-linear self-equivalences of the category~$\calM_R$ and their natural isomorphisms. It remains to be verified that the space of~$R$-linear self-equivalences of the category~$\calM_R$ and their natural isomorphisms is naturally equivalent to the Picard space~$\Picspace(R)$. 

On the level of components, this follows from Morita theory, see~\cite{Bass}. On the level of spaces, the equivalence is given by evaluation at the symmetric monoidal unit~$R$. In more detail, if~$F$ is an~$R$-linear equivalence from~$\calM_R$ to itself, then~$F(R)$ is an invertible~$R$-symmetric~$(R,R)$-bimodule, and these are just the invertible~$R$-modules. If~$F\to G$ is a natural isomorphism between two~$R$-linear self-equivalences, this gives in particular an isomorphism~$F(R)\to G(R)$ betweem the corresponding two invertible~$R$-modules. This map induces the classical isomorphism on components, and the natural automorphisms of the identity are given by the units of~(the center of)~$R$, which are precisely the automorphisms of~$R$ as an~$R$-module.
\end{proof}

The preceding result implies the calculation of all higher homotopy groups of the Brauer space as a corollary. We note that a similar description and computation of the higher Whitehead groups of spaces is out of reach at the moment.

\begin{corollary}\label{cor:pis_of_Br(R)}
 If~$R$ is a commutative ring, then the Brauer space~$\Brspace(R)$ has at most three non-trivial homotopy groups: 
\begin{gather*}
\pi_0\Brspace(R)\cong\Br(R),\\
\pi_1\Brspace(R)\cong\Pic(R),\\
\pi_2\Brspace(R)\cong\GL_1(R),
\end{gather*}
and the results for the Brauer spectrum~$\br(R)$ are the same.
\end{corollary}

\begin{proof}
	The first of these is the isomorphism~\eqref{eq:pi_0_Br}, and the second follows from the preceding theorem together with the isomorphisms~\eqref{eq:pi_0_Pic}. For the third, we only need to recall~\ref{eq:pi_1_Pic}: the fundamental groups of the Picard space~$\Picspace(R)$ are all isomorphic to~$\GL_1(R)$. The final statement follows from the equivalence~$\Brspace(R)\simeq\Omega^\infty\br(R)$.
\end{proof}

We remark that Brauer spectra (and spaces) are rather special in the sense that not every~$2$-truncated connective spectrum is equivalent to the Brauer spectrum of a ring. This follows, for example, from the well-known fact that there is no commutative ring which has exactly five units, whereas there are clearly spaces and spectra such that their second homotopy group has order five.

%The units form a group of order~$5$, so that is necessarily cyclic. Let~$\zeta$ be a generator. Since the order of~$-1$ is either~$1$ or~$2$, and a divisor of~$5$, it must be~$1$, so that the ring would have to have characteristic~$2$. A computation then shows that the element~$x=1+a^2+a^3$ has cube equal to 1. Since 3 is prime to 5, we must have~$x=1$, i.e.~$a^2(a+1)=0$. But then~$a+1=0$, and that is a contradiction.

%%%%%%%%%%%%%%%%%%%%%%%%%%%%%%%%%%%%%%%%%%%%%%%%%%%%%%%%%%%%%%%%%%%%%

\section{A scholion on the Azumaya complex}\label{sec:Duskin}

In this section, we will review some related work of Duskin and Street. Let again~$R$ be an ordinary commutative ring. In this case, Duskin, in~\cite{Duskin}, has built~a reduced Kan complex~$\Az(R)$, the {\it Azumaya complex}, with~$\pi_1\Az(R)$ isomorphic to the Brauer group of~$R$, with the group~$\pi_2\Az(R)$ isomorphic to the Picard group of~$R$, and with the group~$\pi_3\Az(R)$ isomorphic to the multiplicative group of units in~$R$.~(As~$\Az(R)$ is reduced, we may omit the base-point from the notation.) In fact, he hand-crafts the~$4$-truncation so that the homotopy groups work out as stated, and then he takes its~$4$-co-skeleton. 

Here is a sketch of his construction. There is only one~$0$-simplex in~$\Az(R)$. It does not need a name, but it can be thought of as the commutative ring~$R$. We note that also our $s\calA^{fh}_R(0_+)$ is the trivial category. The~$1$-simplices in the Azumaya complex~$\Az(R)$ are the Azumaya~$R$-algebras~$A$. (In particular, the degenerate~$1$-simplex is given by the~$R$-algebra~$R$ itself.) We note that these are precisely the objects in our category~\hbox{$s\calA^{fh}_R(1_+)\simeq s\calA^{fh}_R$}; but, the latter comes with higher homotopy information form the mapping spaces. Now a map $\partial\Delta^2\to\Az(R)$ corresponds to three Azumaya algebras $A_1$, $A_2$, and $A_{12}$, and the $2$-simplices in $\Az(R)$ with this given restriction are defined to be the~$R$-symmetric~$(A_{12},A_1\otimes_RA_2)$-bimodules $F$ which are invertible.~A suggestive notation is~\hbox{$F\colon A_{12}\Rightarrow A_1\otimes_RA_2$}. By Example~\ref{ex:gamma012}, these are essentially the objects of the category $s\calA^{fh}_R(2_+)$, except for the fact that we are working with the actual equivalences defined by the bimodules. Now a map $\partial\Delta^3\to\Az(R)$ corresponds to four bimodules
\begin{align*}
	A_{123}&\Longrightarrow A_{12}\otimes_RA_3\\
	A_{123}&\Longrightarrow A_1\otimes_RA_{23}\\
	A_{12}&\Longrightarrow A_1\otimes_RA_2\\
	A_{23}&\Longrightarrow A_2\otimes_RA_3,
\end{align*}
and the $4$-simplices in $\Az(R)$ with this boundary are the isomorphisms $u$ between the two corresponding bimodules
\[
	A_{123}\Longrightarrow A_1\otimes_RA_2\otimes_RA_3
\]
that can be obtained by tensoring them in the two meaningful ways. Finally, the $4$-simplices of $\Az(R)$ are uniquely determined by their boundary, and their existence depends on a compatibility condition that we will not recall here.

\begin{remark}\label{rem:Street}
As already mentioned in~\cite{Duskin}, Street has described some catego\-ri\-cal structures underlying Duskin's construction. However, these were published only much later, in~\cite{Street}. Street considers the bicategory whose objects are~$R$-algebras, whose morphism~$M\colon A\to B$ are~$R$-symmetric~$(A,B)$-bimodules, and whose~$2$-cells~\hbox{$f\colon M\Rightarrow N$} are bimodule morphisms; vertical composition is composition of functions and horizontal composition of modules~$M\colon A\to B$ and~\hbox{$N\colon B\to C$} is given by tensor product~$M\otimes_B N\colon A\to C$ over~$B$. The tensor product~$A\otimes_R B$ of algebras is again an algebra, and this makes this category a monoidal bicategory. He then passes to its suspension, the one-object tricategory whose morphism bicategory is the category described before and whose composition is the tensor product of algebras. While this cannot, in general, be rigidified to a 3-category, there is a~$3$-equivalent~Gray category. The Gray subcategory of invertibles consists of the arrows~$A$ which are biequivalences, the~$2$-cells~$M$ which are equivalences, and the~$3$-cells~$f$ which are isomorphisms, so that the morphisms~$A$ are the Azumaya algebras, and the~$2$-cells are the Morita equivalences. The nerve of this Gray subcategory is Duskin's complex~$\Az(R)$.
\end{remark}

In this paper, we have chosen to present an approach that does not involve higher categories, at least none that do not have a well-defined composition. While one may argue that the loop space $\Omega\Az(R)$ would be equivalent to the Brauer space $\Brspace(R)$, the present direct construction seems to be more natural. It certainly seems rather artificial to realize the Brauer group as $\pi_1$ instead of $\pi_0$. In any case, our delooping~$|s\calA^{fh}_R(\upS^1)|$ provides for a space with such a $\pi_1$ as well, if so desired, see Proposition~\ref{prop:delooping_rings}. In fact, the general algebraic K-theory machinery used here yields arbitrary deloopings $|s\calA^{fh}_R(\upS^n)|$ without extra effort. This feature seems to be unavailable in the approach of Duskin and Street.

\section{Brauer spectra for structured ring spectra}\label{sec:Brauer_for_commutative_S-algebras}

We will now transfer the preceding theory from the context of commutative rings to the context of  structured ring spectra. There are many equivalent models for this, such as symmetric spectra~\cite{HSS} or~$\Sbb$-modules~\cite{EKMM}, and we will choose the latter for the sake of concordance with~\cite{Baker+Richter+Szymik}. In Section~\ref{sec:Brauer_for_commutative_rings}, we have defined a Brauer space~$\Brspace(R)$ and a Brauer spectrum~$\br(R)$ for each commutative ring~$R$, starting from a category~$\calA_R$ and its subcategory~$s\calA^{fh}_R$. If now~$R$ denotes a commutative~$\Sbb$-algebra, we may proceed similarly. Let us see how to define the corresponding categories.

\subsection*{The categories~$\calA_R$}

Let~$\calA_R$ denote the category of cofibrant~$R$-algebras and~$R$-functors between their categories of modules. This is slightly more subtle than the situation for ordinary rings, as the categories of modules are not just categories, but come with homotopy theories. In order to take this into
account, the model categories of modules will first be replaced by the simplicial categories obtained from them by Dwyer-Kan localization. We note that the categories of modules are enriched in the symmetric monoidal model category of~$R$-modules in this situation. This allows us to use the model structure from~\cite[Appendix A.3]{Lurie:HTT} on these. Then~$\calA_R$ is again a simplicial category: The class of objects~$A$ is still discrete, and the space of morphisms~$A\to B$ is the derived mapping space of~$R$-functors~$\calM_A\to\calM_B$.

\subsection*{Decorated variants}

There is the full subcategory~$\calA^f_R$, where~$A$ is assumed to satisfy the finiteness condition used in~\cite{Baker+Richter+Szymik}: it has to be faithful and dualizable as an~$R$-module.

Also, there is the full subcategory~$\calA^h_R$, where we assume that the natural map
\begin{equation*}
	A\wedge_RA^\circ\longrightarrow\End_R(A)
\end{equation*}
is an equivalence. We are mostly interested in the intersection
\begin{equation*}
	\calA^{fh}_R=\calA^f_R\cap\calA^h_R,
\end{equation*}
which consists precisely the Azumaya~$R$-algebras in the sense of~\cite{Baker+Richter+Szymik}. While~$f$ and~$h$ again refer to restrictions on the objects, the prefix~$s$ will indicate that we are considering only those~$R$-functors which are equivalences of simplicial categories, and their natural equivalences. The standard references for Morita theory in this context are~\cite{Schwede+Shipley} as well as the expositions~\cite{Schwede} and~\cite{Shipley}. Up to natural equivalence, the~$R$-equivalences are of the form~\hbox{$X\longmapsto X\wedge_AM$} for some invertible~$R$-symmetric~$(A,B)$-bimodule~$M$. Similarly, the higher simplices codify natural equivalences rather than all natural transformations. This ends the description of the simplicial category~$s\calA^{fh}_R$. The following result describes one of its mapping spaces.

\begin{proposition}\label{prop:auto_R}
	The space of auto-equivalences of the category of~$R$-modules is naturally equivalent to the space of invertible~$R$-modules.
\end{proposition}

\begin{proof}
	This is formally the same as the corresponding result in the proof of Theorem~\ref{thm:Pic_identification_rings}. A map in one direction is given by the evaluation that sends an equivalence to its value on~$R$. On the other hand, given an invertible~$R$-module, the smash product with it defines an equivalence. Compare with~\cite[4.1.2]{Schwede+Shipley}.
\end{proof}

%%%%%%%%%%%%%%%%%%%%%%%%%%%%%%%%%%%%%%%%%%%%%%%%%%%%%%%%%%%%%%%%%%%%%

\subsection*{A symmetric monoidal structure}

A symmetric monoidal structure on~$\calA_R$ and its subcategories is induced by the smash product
\begin{equation*}
	(A,B)\mapsto A\wedge_RB
\end{equation*}
of~$R$-algebras with neutral element~$R$. We note that this is not the categorical sum, as the morphisms in these categories are not just the algebra maps.

\begin{proposition}\label{prop:Br_is_a_group_2}
	With the induced multiplication, the abelian monoid~$\pi_0|s\calA^{fh}_R|$ of isomorphism classes of objects is an abelian group, and this abelian group is isomorphic to the Brauer group~$\Br(R)$ of the commutative~$\Sbb$-algebra~$R$ in the sense of~{\upshape\cite{Baker+Richter+Szymik}}.
\end{proposition}

\begin{proof}
	This is formally the same as the proofs of Proposition~\ref{prop:pi0_is_a_group_1} and Proposition~\ref{prop:pi0_is_Brauer_1}.
\end{proof}

%%%%%%%%%%%%%%%%%%%%%%%%%%%%%%%%%%%%%%%%%%%%%%%%%%%%%%%%%%%%%%%%%%%%%

\subsection*{Brauer spaces and spectra}

The following definition is suggested by Proposition~\ref{prop:Br_is_a_group_2}.

\begin{definition}
Let~$R$ be a commutative~$\Sbb$-algebra. The space
\begin{equation*}
	\Brspace(R)=|s\calA^{fh}_R|
\end{equation*}
is called the {\it Brauer space} of~$R$. 
\end{definition}

By Proposition~\ref{prop:Br_is_a_group_2}, there is an isomorphism
\begin{equation}
	\pi_0\Brspace(R)\cong\Br(R)
\end{equation}
that is natural in~$R$.

As described in Section~\ref{sec:background}, the symmetric monoidal structure on~$s\calA^{fh}_R$ also gives rise to an algebraic K-theory spectrum.

\begin{definition}
	Let~$R$ be a commutative~$\Sbb$-algebra. The spectrum
	\begin{equation*}
		\br(R)=\bfk(s\calA^{fh}_R)
	\end{equation*}
	is called the {\it Brauer spectrum} of~$R$. 
\end{definition}

We will now spell out the relation between the Brauer space~$\Brspace(R)$ and the Brauer spectrum~$\br(R)$ in detail.

\begin{proposition}
	 There are natural homotopy equivalences
	 \begin{equation*}
		\Omega^\infty\br(R)
		\simeq\Omega|s\calA^{fh}_R(\upS^1)|
		\simeq|s\calA^{fh}_R(\upS^0)|
		\simeq|s\calA^{fh}_R|=\Brspace(R).
	\end{equation*}
\end{proposition}

\begin{proof}
As explained in Section~\ref{sec:background}, the first and third of these generally hold for symmetric monoidal categories, and the second uses the additional information provided by Proposition~\ref{prop:Br_is_a_group_2}, namely that the spaces on the right hand side are already group complete.
\end{proof}

In particular, we also have natural isomorphisms
\[
	\pi_0\br(R)\cong\Br(R),
\]
of abelian groups, so that we can recover the Brauer group as the~$0$-th homotopy group of a spectrum. We will now determine the higher homotopy groups thereof.

%%%%%%%%%%%%%%%%%%%%%%%%%%%%%%%%%%%%%%%%%%%%%%%%%%%%%%%%%%%%%%%%%%%%%

\subsection*{A review of Picard spaces and spectra}

As for the Brauer space of a commutative ring, also in the context of structured ring spectra, there is a relation to the Picard groupoid of~$R$, and this will be discussed now. 

Let~$R$ be a cofibrant commutative~$\Sbb$-algebra. In analogy with the situation for discrete rings, it is only natural to make the following definition, compare~\cite[Remark~1.3]{Ando+Blumberg+Gepner}.

\begin{definition}
	The {\it Picard space}~$\Picspace(R)$ of~$R$ is the classifying space of the Picard groupoid: the simplicial groupoid of invertible~$R$-modules and their equivalences.
\end{definition}

We note that the components of the Picard space are the equivalence classes of invertible~$R$-modules, and with respect to the smash product~$\wedge_R$, these components form a group, by the very definition of an invertible module. This is Hopkins' Picard group~$\Pic(R)$ of~$R$:
\begin{displaymath}
	\pi_0(\Picspace(R))\cong\Pic(R).
\end{displaymath}
See for example~\cite{Strickland},~\cite{Hopkins+Mahowald+Sadofsky}, and~\cite{Baker+Richter} for more information about this group. In contrast to the case of discrete commutative rings, the Picard space of a commutative~$\Sbb$-algebra is no longer $1$-truncated. There is an equivalence
\[
	\Picspace(R)\simeq\Pic(R)\times\upB\GL_1(R),
\]
where $\upB\GL_1(R)$ is the classifying space for the units of $R$.

The Picard category is symmetric monoidal with respect to the smash product~$\wedge_R$. Therefore, as in the case of discrete rings, there is also a Picard spectrum~$\pic(R)$ such that~$\Picspace(R)\simeq\Omega^\infty\pic(R)$, and there is a connected delooping
\[
	\upB\Picspace(R)\simeq\Omega^\infty\Sigma\pic(R)
\]
such that its homotopy groups are isomorphic to those of~$\Picspace(R)$, but shifted up by one. We will see now that the Brauer spaces provide for another delooping that is typically non-connected.

%\begin{remark}
%Units in ring spectra are related to Thom spectra, because Thom spectra can be defined from maps~$X\rightarrow\BGL_1(R)$. This is classical for~$R=\Sbb$, but otherwise needs a modern version of the Thom spectrum construction, such as given in \cite{Ando+Blumberg+Gepner+Hopkins+Rezk}. See also \cite{Ando+Hopkins+Rezk}. If~$R$ is an~$\Sbb$-algebra, and~$f\colon X\rightarrow \BGL_1(R)$ is a map, there is 
%an~$R$-module Thom spectrum~$\M(f)$. The{\it~$f$-twisted~$R$-homology of~$X$} is by definition
%\begin{displaymath}
%	R^f_k(X)=\pi_0(\calM_R(\Sigma^kR,\M(f))),
%\end{displaymath}
%while the{\it~$f$-twisted~$R$-cohomology of~$X$} is
%\begin{displaymath}
%	R_f^k(X)=\pi_0(\calM_R(\M(f),\Sigma^kR)).
%\end{displaymath}
%If~$f$ factors through~$\BGL_1(\Sbb)\rightarrow\BGL_1(R)$, then the~$R$-module Thom spectrum is just the base change to~$R$ of the ordinary~$\Sbb$-module Thom spectrum, so that in this case, the~$f$-twisted homology and cohomology of~$X$ coincides with the untwisted~$R$-homology and cohomology of the ordinary~$\Sbb$-module Thom spectrum of the spherical fibration. Otherwise, the~$f$-twisted generalized cohomology is the cohomology of the more general~$R$-module Thom spectrum. More generally, one may now use maps~$X\rightarrow\Picspace(R)=\cat{R}{-inv}{}$ instead of those into~$\BGL_1(R)=\cat{R}{-line}{}$ throughout this story.
%\end{remark}

\subsection*{Higher homotopy groups}

After this recollection, let us now see how the Picard spaces and spectra relate to the Brauer spaces and spectra defined above.

\begin{theorem}\label{thm:Pic_identification_S-algebras}
	If~$R$ is a commutative~$\Sbb$-algebra, then the component of the neutral vertex~$R$ in the Brauer space~$\Brspace(R)$ is naturally equivalent (as an infinite loop space) to the classifying space of the Picard groupoid~$\Picspace(R)$.
\end{theorem}

\begin{proof}
	The component of the neutral element~$R$ in~$\Brspace(R)=|s\calA^{fh}_R|$ is equivalent to the classifying space of the automorphism group of $R$ in $s\calA^{fh}_R$. By definition of that category, this is the group-like simplicial monoid of Morita self-equivalences of the category $\calM_R$ of $R$-modules. The result now follows from Proposition~\ref{prop:auto_R}: this is equivalent to the simplicial groupoid of invertible $R$-modules, the Picard groupoid of $R$.
\end{proof}

\begin{corollary}\label{cor:Pic_identification}
	There are natural isomorphisms
	\begin{equation*}
		\pi_n\Brspace(R)\cong\pi_{n-2}\Picspace(R)
	\end{equation*}
for~$n\geqslant2$, 
\begin{equation*}
	\pi_n\Brspace(R)\cong\pi_{n-3}\GL_1(R)
\end{equation*}
for~$n\geqslant3$, and 
\begin{equation*}
	\pi_n\Brspace(R)\cong\pi_{n-3}(R).
\end{equation*}
for~$n\geqslant4$. 
\end{corollary}

\begin{proof}
	The first statement is an immediate consequence of the preceding theorem. The second follows from the first and the fact that the Picard space is a delooping of the space of units, and the last statement follows from the second and~$\pi_n\GL_1(R)\cong\pi_n(R)$ for~\hbox{$n\geqslant1$}.
\end{proof}

We note in particular that the Brauer space is~$2$-truncated in the case of an Eilenberg-Mac Lane spectrum. This will be used in~Section~\ref{sec:EM}.

%%%%%%%%%%%%%%%%%%%%%%%%%%%%%%%%%%%%%%%%%%%%%%%%%%%%%%%%%%%%%%%%%%%%%

\section{Functoriality} 

In this section, we will see how functorality of the K-theory construction immediately leads to spacial and spectral versions of relative Brauer invariants as well as to a characteristic map that codifies the obstructions to pass from topological information about Eilenberg-Mac Lane spectra to algebra.

\subsection*{Relative invariants}\label{sec:relative}

In this section, we define relative Brauer spectra, as these are likely to be easier to compute than their absolute counterparts. We will focus on the case of extensions of commutative~$\Sbb$-algebras, but the case of ordinary commutative rings is formally identical.

To start with, let us first convince ourselves that the construction of the Brauer space (or spectrum) is sufficiently natural.

\begin{proposition}\label{prop:naturality}
	If~$R\to S$ is a map of commutative~$\Sbb$-algebras, then there is a map
	\begin{equation*}
		\br(R)\longrightarrow\br(S)
	\end{equation*}
	of Brauer spectra, and similarly for Brauer spaces.
\end{proposition}

\begin{proof}
	The map is induced by~$A\mapsto S\wedge_RA$. By~\cite[Proposition 1.5]{Baker+Richter+Szymik}, it maps Azumaya algebras to Azumaya algebras. It therefore induces functors between the symmetric monoidal categories used to define the Brauer spectra. 
\end{proof}

If~$R\to S$ is a map of commutative~$\Sbb$-algebras, then~$\br(S/R)$ and~$\Brspace(S/R)$ will denote the homotopy fibers of the natural maps in Proposition~\ref{prop:naturality}. We note that there is an equi\-valence~\hbox{$\Brspace(S/R)\simeq\Omega^\infty\br(S/R)$} of infinite loop spaces.

\begin{remark}\label{rem:Thomason}
Thomason's general theory of homotopy colimits of symmetric monoidal categories (\cite{Thomason:Aarhus} and \cite{Thomason:CommAlg}) might provide a first step to obtain a more manageable description of these relative terms.
\end{remark}

The defining homotopy fibre sequences lead to exact sequences of homotopy groups. Together with the identifications from Proposition~\ref{thm:Pic_identification_S-algebras} and Corollary~\ref{cor:Pic_identification}, these read
\begin{gather*}
	\dots\to\pi_2\br(S/R)\to\pi_0\GL_1(R)\to\pi_0\GL_1(S)\to
	\pi_1\br(S/R)\to	\Pic(R)\to\\
	\to\Pic(S)\to
	\pi_0\br(S/R)\to\Br(R)\to\Br(S)\to\pi_{-1}\br(S/R)\to0.
\end{gather*}

In~\cite[Definition 2.6]{Baker+Richter+Szymik}, the relative Brauer group~$\Br(S/R)$ is defined as the kernel of the natural homomorphism
$\Br(R)\to\Br(S)$.

\begin{proposition}\label{prop:relative_relation}
	The relative Brauer group~$\Br(S/R)$ is naturally isomorphic to the cokernel of the natural boundary map~$\Pic(S)\to\pi_0\br(S/R)$.
\end{proposition}

\begin{proof}
	This is an immediate consequence of the definition of~$\Br(S/R)$ as the kernel of the natural homomorphism~$\Br(R)\to\Br(S)$ and the long exact sequence above.
\end{proof}

\begin{remark}
	We note that the theory of Brauer spaces presented here also has Bousfield local variants, building on~\cite[Definition 1.6]{Baker+Richter+Szymik}, and that it might be similarly interesting to study the behavior of the Brauer spaces under variation of the localizing homology theory.
\end{remark}

%%%%%%%%%%%%%%%%%%%%%%%%%%%%%%%%%%%%%%%%%%%%%%%%%%%%%%%%%%%%%%%%%%%%%

\subsection*{Eilenberg-Mac Lane spectra}\label{sec:EM}

Let us finally see how the definitions of the present paper work out in the case of Eilenberg-Mac Lane spectra. Let~$R$ be an ordinary commutative ring, and let~$\upH R$ denote its Eilenberg-Mac Lane spectrum. This means that we have two Brauer groups to compare:~$\Br(R)$ as defined in~\cite{Auslander+Goldman}, and~$\Br(\upH R)$ as defined in~\cite{Baker+Richter+Szymik}, where there is also produced a natural homomorphism
\begin{equation}\label{eq:EMmap}
	\Br(R)\longrightarrow\Br(\upH R)
\end{equation}
of groups, see~\cite[Proposition 5.2]{Baker+Richter+Szymik}. Using results from~\cite{Toen}, one can deduce that this homomorphism is an isomorphism if~$R$ is a separably closed field, because both sides are trivial, see~\cite[Proposition 5.5 and Remark 5.6]{Baker+Richter+Szymik}. In general, one may show that it is injective with a cokernel which is isomorphic to the product of~$\upH^1_\mathrm{et}(R;\ZZ)$ and the torsion-free quotient of~$\upH^2_\mathrm{et}(R;\GL_1)$. See~\cite{Johnson} and~\cite[Remark 5.3]{Baker+Richter+Szymik}. In particular, the map is an isomorphism for all fields~$K$. 

Using the spaces and spectra defined in~Section~\ref{sec:Brauer_for_commutative_rings} for~$R$ and in Section~\ref{sec:Brauer_for_commutative_S-algebras} for~$\upH R$, the homomorphism~\eqref{eq:EMmap} of abelian groups can now be refined to a map of spectra.

\begin{proposition}\label{prop:EMmap_spectra}
There is a natural map
\begin{equation}\label{eq:EMmap_spectra}
	\br(R)\longrightarrow\br(\upH R)
\end{equation}
of spectra that induces the homomorphism~\eqref{eq:EMmap} on components.
\end{proposition}

\begin{proof}
	This map is induced by the Eilenberg-Mac Lane functor~$\upH\colon R\mapsto\upH R$. It induces functors between the symmetric monoidal categories used to define the Brauer spectra. 
\end{proof}

\begin{theorem}\label{prop:1-truncated}
	The homotopy fibre of the map~\eqref{eq:EMmap_spectra} is a~$0$-truncated spectrum. Its only non-trivial homotopy groups are~$\pi_0$ which is infinite cyclic, and~$\pi_{-1}$ which is isomorphic to the cokernel of the map~\eqref{eq:EMmap}, the product of~$\upH^1_\mathrm{et}(R;\ZZ)$ and the torsion-free quotient of~$\upH^2_\mathrm{et}(R;\GL_1)$.
\end{theorem}

\begin{proof}
	If~$R$ is a commutative ring, then the natural equivalence
	\begin{displaymath}
		\gl_1(\upH R)\simeq\upH\GL_1(R)
	\end{displaymath}
	describes the spectrum of units of the Eilenberg-Mac Lane spectrum. It follows that~$\br(\upH R)$ is~$2$-truncated. As~$\br(R)$ is always~$2$-truncated by Corollary~\ref{cor:pis_of_Br(R)}, so is the homotopy fibre. On~$\pi_2$, the map~\eqref{eq:EMmap_spectra} induces an isomorphism between two groups both isomorphic to the group of units of~$R$. On~$\pi_1$, the map~\eqref{eq:EMmap_spectra} is the map
\begin{equation}\label{eq:Pic_map}
	\Pic(R)\longrightarrow\Pic(\upH R)
\end{equation}
induced by the Eilenberg-Mac Lane functor~$\upH$. A more general map has been studied in~\cite{Baker+Richter}, where the left hand side is replaced by the Picard group of graded~$R$-modules, and then the map is shown to be an isomorphism. For the present situation, this means that~\eqref{eq:Pic_map} is a monomorphism with cokernel isomorphic to the group~$\ZZ$ of integral grades. On~$\pi_0$, the map~\eqref{eq:EMmap_spectra} induces the map~\eqref{eq:EMmap} by Proposition~\ref{prop:EMmap_spectra}. As has been remarked at the beginning of this section, this map is injective with the indicated cokernel. The result follows. 
\end{proof}

%%%%%%%%%%%%%%%%%%%%%%%%%%%%%%%%%%%%%%%%%%%%%%%%%%%%%%%%%%%%%%%%%%%%

%%%%%%%%%%%%%%%%%%%%%%%%%%%%%%%%%%%%%%%%%%%%%%%%%%%%%%%%%%%%%%%%%%%%%

\vfill

{\it Affiliation:}\\
Mathematisches Institut\\Heinrich-Heine-Universit\"at\\40225 D\"usseldorf\\Germany

{\it Current address:}\\
Department of Mathematical Sciences\\University of Copenhagen\\2100 Copenhagen~\O\\Denmark

{\it Email address:}\\
szymik@math.ku.dk

\end{document}